\documentclass[a4paper,10pt,leqno, spanish]{amsart}
\title[Trees]
{Trees are $1$-Transfer }

\author{Salvador Sierra Murillo }
        \email{ssierra.murillo@gmail.com}

         \date{\today}

\usepackage{hyperref}
\usepackage{pdfsync}
\usepackage{calc}
\usepackage{enumerate,amssymb}
\usepackage[arrow,curve,matrix,tips,2cell]{xy}
  \SelectTips{eu}{10} \UseTips
  \UseAllTwocells

\usepackage{graphicx}
\usepackage{float}

\DeclareMathAlphabet\EuR{U}{eur}{m}{n}
\SetMathAlphabet\EuR{bold}{U}{eur}{b}{n}

\makeindex             

\theoremstyle{plain}
\newtheorem{theorem}{Theorem}[section]

\newtheorem{proposition}[theorem]{Proposition}

\newtheorem{conjecture}[theorem]{Conjecture}

\theoremstyle{definition}
\newtheorem{definition}[theorem]{Definition}
\newtheorem{example}[theorem]{Example}

\newtheorem{remark}[theorem]{Remark}

{\catcode`@=11\global\let\c@equation=\c@theorem}

\newcommand{\calf}{\mathcal{F}}

\newcommand{\calfcyc}{\mathcal{FCY}}

\newcommand{\calfin}{\mathcal{FIN}}

\newcommand{\calvcyc}{\mathcal{VCY}}

\newcommand{\IN}{{\mathbb N}}
\newcommand{\IR}{{\mathbb R}}

\newcommand{\calc}{{\mathcal C}}

\newcommand{\bfK}{{\mathbf K}}

\begin{document}

\begin{abstract}
The K-theoretic Farrell-Jones isomorphism conjecture for a group ring $R[G]$ has been proved for several groups. The toolbox for proving the Farrell-Jones conjecture for a given group depends on some geometric properties of the group as it is the case of hyperbolic groups. The technique used to prove it for hyperbolic groups $G$ relies in the concept of an $N$-transfer space endowed with a $G$ action. In this work, we give an explicit construction of a $1$-transfer space.
\end{abstract}

\maketitle

\section{Introduction}

The present work examine one tool using to prove the $K$-theoretic Farrell-Jones Isomorphism Conjecture. The conjecture itself is not the aim of this work. However, to give our result the right dimension we start with a short review of the formulation and results on the Farrell-Jones conjecture.

Let $G$ be a group. A family of subgroups of $G$ is a non-empty collection $\calf$ of subgroups that is closed under conjugation and taking subgroups.

    \begin{example}[Family of subgroups]
    Let $G$ be a group. Useful examples of families of subgroups of $G$ are the family of finite subgroups $\calfin$, the family of cyclic subgroups $\calfcyc$ and the family of virtually cyclic subgroups $\calvcyc$.
    \end{example}

    \begin{definition}
    A $G$-CW-Complex $E$ is called a classifying space for the family $\calf$ of subgroups of $G$, if $E^H$ (fixed points) is always contractible for all $H\in\calf$ and empty otherwise.
    \end{definition}

It is a well-known result that the for any family $\calf$ always exists a classifying space for the family up to $G$-equivariant homotopy. A model for the classifying space of a family is usually denoted by $E_{\calf}G$.

Following the construction of \cite{MR1659969}, given a ring $R$ and a group $G$, they construct a homology theory for $G$-spaces
    \begin{equation*}
    X \longrightarrow H_{*}^{G}(X;\bfK_R)
    \end{equation*}
with the property $H_{*}^{G}(G/H;\bfK_R)=K_*(R[H])$.

    \begin{definition}[$\calf$-assembly map]
    Let $\calf$ be a family of subgroups of $G$. The projection $E_{\calf}G\twoheadrightarrow G/G$ to the one-point $G$-space induces the $\calf$-assembly map
        \begin{equation*}
        \alpha_{\calf}:H_{*}^{G}(E_{\calf}G;\bfK_R)\rightarrow H_{*}^{G}(G/G;\bfK_R)=K_*(R[G])
        \end{equation*}
    \end{definition}

    \begin{conjecture}[Farrell-Jones]
    For all groups $G$ and all rings $R$, the assembly map $\alpha_{\calvcyc}$ is an isomorphism.
    \end{conjecture}

The Farrell-Jones conjecture happens to be true for a large class of groups. Examples for which the conjecture is true are the family of Hyperbolic groups \cite{Bartels2008} and Fundamental groups of graphs of Virtually Cyclic groups \cite{Wu2016}. Also, all of them illustrate how diverse the techniques to prove this conjecture are.

While it is still an open conjecture and the proves does not show a recognizable pattern some attempts to find general arguments has been done. On Proofs of the Farrell–Jones Conjecture \cite{10.1007/978-3-319-43674-6_1} the statement of Theorem A has this objective. The statement is

    \begin{theorem}[Theorem A, \cite{10.1007/978-3-319-43674-6_1}] 
    Suppose $G$ is finitely generated by $S$. Let $\calf$ be family of subgroups of $G$. Assume that there is $N\in\IN$ such that for any $\epsilon>0$ there are
        \begin{enumerate}
            \item [(a)] an $N$-transfer space $X$ equipped with a $G$-action,
            \item [(b)] a simplicial $(G,\calf)$-complex $E$ of dimension at most $N$,
            \item [(c)] a map $f:X\to E$ that is $G$-equivariant up to $\epsilon$: $d'(f(sx),sf(x))\leq\epsilon$ for all $s\in S$, $x\in X$.
        \end{enumerate}
    Then $\alpha_{\calf}:H_{*}^{G}(E_{\calf}G;K_R)\rightarrow K_*(R[G])$ is an isomorphism
    \end{theorem}

A simplicial $(G,\calf)$-complex a simplicial complex $E$ with a simplicial $G$-action whose isotropy groups $G_x=\{g\in G|gx=x\}$ belongs to $\calf$ for all $x\in E$.

    \begin{remark}
    Theorem A applies to Hyperbolic groups and $\calf=\calvcyc$. the family of virtually cyclic sugroups. This theorem is a minor formulation of [Bartels-Lueck-Reich] work.
    \end{remark}

The concept of $N$-transfer is our main concern. In [Bar], an example of a $1$-transfer space is given. Namely, the compactification $\overline{T}$ of a locally finite simplicial tree $T$ by geodesic rays. The result is elementary but by no means trivial. There are no explicit prove of this result and the author provides a proof of this fact.

\section{Preliminaries}

In this section, we give the necessary definitions to establish our result.

    \begin{definition}[$N$-transfer]
    An $N$-transfer space $X$ is a compact contractible metric space such that the following holds.
    For any $\delta>0$ there exists a simplicial complex $K$ of dimension at most $N\in\IN$, continuous maps $i:X\to K$, $p:K\to X$ and homotopy $H:p\cdot i\to Id_X$ such that for any $x\in X$, $\text{diam}\{H(t,x)|t\in[0,1]\}\leq\delta$.
    \end{definition}

The definition of an $N$-transfer recall us that of a dominated complex. The difference is the control requirement on the diameter being less or equal to $\delta$.

    \begin{definition}[Abstract simplicial complex]
    An abstract simplicial complex $K$ consists of a non-empty set $V$ of vertices and a collection $\mathcal{S}$ of non-empty subsets of $V$ such that
        \begin{itemize}
            \item for every $v\in V$, $\{v\}\in\mathcal{S}$,
            \item if $S\in\mathcal{S}$, then $\emptyset\neq T\subseteq\mathcal{S}$ is in $\mathcal{S}$. 
        \end{itemize}
    \end{definition}

    \begin{remark}
    We call the elements of $\mathcal{S}$ the simplices of $K$. An element $S\in\mathcal{S}$ is an $n$-simplex if $|S|=n+1$ and we set $dim(S)=n$. The single elements of $n$-simplex $S$ are called vertices and each proper subset of $S$ is a face.
    \end{remark}
    
    \begin{definition}[Simplicial Tree]
    A simplicial tree $T$ is a connected simply connected $1$-simplex.
    \end{definition}
    
    \begin{remark}
    Working with an abstract simplicial complex whose simplex are $1$-simplex might be undue. Instead, we use its geometric realization. Every statement we made on the geometric realization of $T$ is valid for the abstract simplicial tree.
    \end{remark}

    \begin{definition}[Simplicial Metric]
    A metric $d$ on a simplicial complex $K$ is called simplicial if
        \begin{itemize}
            \item [(a)] the restriction, $d_i=d|_{S_i}$, to each simplex $S_i$ is euclidean,
            \item [(b)] $d$ is maximal for the condition $d_i=d|_{S_i}$ for each $S_i$.
        \end{itemize}
    \end{definition}

    \begin{figure}[H]
        \centering
        \includegraphics[scale=0.3]{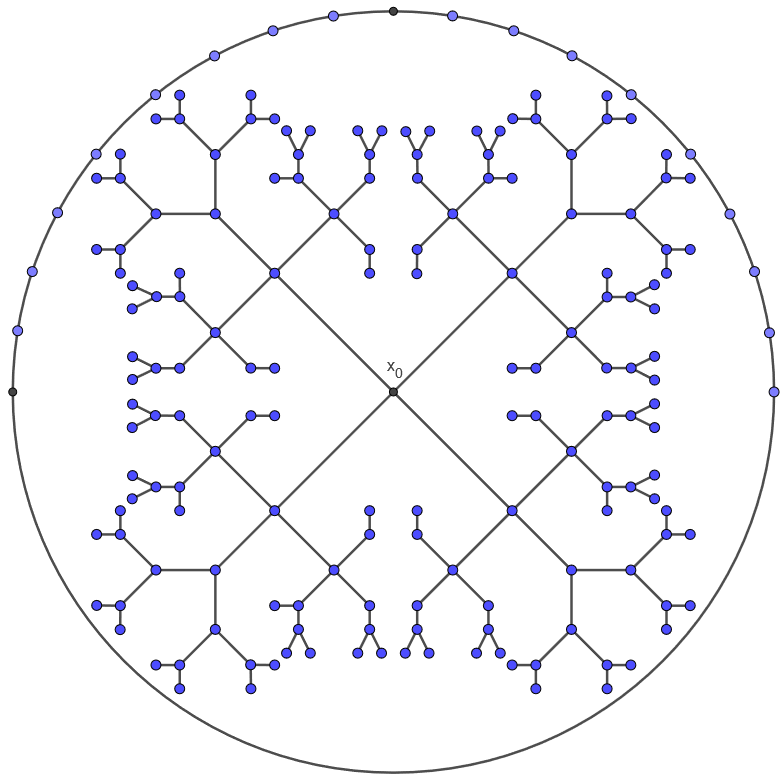}
        \caption{Example of a geometric realization $T$ of a Tree and its compactification $\overline{T}$}
        \label{figure1}
    \end{figure}

We give a simplicial metric to a tree $T$ as follows.

Let $x_0\in T$ be an arbitrary fixed vertex of $T$. A vertex $y\in T$ is adjacent to $x_0$ if it is joined with $x_0$ by just one edge. Give length $1/2$ to all edges joining adjacent $y$ to $x_0$. For the next step, consider all the adjacent vertices to $x_0$ and give length $1/4$ to all edges joining adjacent vertices to $y$ but that joining $x_0$ and $y$ which has already length $1/2$.

In general, an edge between $x_1$ and $x_2$ adjacent vertices has length $1/2^{n}$ if the (minimum) number of vertices from $x_1$ to $x_0$ is $n-1$ and the number of vertices from $x_2$ to $x_0$ is $n$. 

Our goal is to define a metric on the geometric realization of $T$. Call a pair of vertices $x,y\in T$ adjacent if they are joined by just one edge. We now give each edge on the geometric realization of $T$ a length.

    \begin{enumerate}
        \item Let $x_0\in T$ be an arbitrary but fixed vertex.
        \item The length of an edge joining $x_0$ with an adjacent vertex $y$ is $1/2$.
        \item Recursively, an edge joining $y$, adjacent to $x_0$, with any other adjacent vertex different from $x_0$ has length $1/2^2$.
        \item Repeat steps (ii) and (iii).
    \end{enumerate}
    
In general, a vertex $x\in T$ can be reached from $x_0$ by a minimum sequence of vertices $x_1,x_2,...,x_n$ with $x_i$ adjacent to $x_{i+1}$ and $x_n$ adjacent to $x$. By definition of lengths it is easy to see that the edge joining $x$ with $x_n$ has length $1/2^{n+1}$.

Defining the length of each edge as we did induces a metric $d$ on $T$. Indeed, this metric makes $T$ into a geodesic metric space and, by definition of a tree, into a uniquely geodesic metric space. We denote this metric space by $(T,d)$.

    \begin{proposition}\label{gromov}
    The metric tree $(T,d)$ is an $\IR$-tree. Moreover, (T,d) is a $CAT(k)$-spaces for every $k\in\IR$.
    \end{proposition}

The proof of Proposition \ref{gromov} and precise definition of $\IR$-tree are in Gromov's work \cite{Gromov1987}. For us it is enough to consider $(T,d)$ as $CAT(0)$-space because the compactification we consider depends only on the properties of $CAT(0)$-spaces. We refer the reader to (\cite{Bartels2008} Chap II. 8) for a detailed exposition of this proof.

    \begin{definition}[Gromov Product]\label{Grom}
    Let $(T,d)$ be a metric space with a distinguished point $x_0$ and $d(x)$ denote $d(x,x_0)$. The Gromov product is
        \begin{equation*}
            (x|y)=\frac{1}{2}(d(x)+d(y)-d(x,y))
        \end{equation*}
    \end{definition}

We can think of the Gromov product as a way to measure the distance from $x_0$ to the vertex joining $x$ to $y$ in $T$. It is straightforward to see that $(T,d)$ is a $0$-hyperbolic space, that is, for every $x,y,z\in T$ and fixed $x_0$ we have

    \begin{equation*}
        (x|y)\geq min\{(x|z),(z|y)\}.
    \end{equation*}

    \begin{definition}[Geodesic ray and convergence]
    A geodesic ray $C:[0,\infty]\rightarrow T$ (based at $x_0$) is a sequence $\{x_i\}$ of vertices such that each of them belongs to the image of $C$. Let $\{x_i\}$ be a sequence of points in $(T,d)$. We say that $\{x_i\}\rightarrow\infty$ converges to the infinity if $(x_i|x_j)\rightarrow 1$ for $i,j\rightarrow\infty$ 
    \end{definition}

Consider the set $\calc$ of all sequences that converges to the infinity in a simplicial metric tree. A pair $\{x_i\}$, $\{y_i\}$ in $\calc$ is equivalent if

    \begin{equation}\label{rel}
        \lim_{i,j\to\infty}inf\{(x_i|y_j)\}=1
    \end{equation}

This equivalence defines an equivalence relation in $\calc$ since $(T,d)$ is $0$-hyperbolic.

    \begin{definition}[The boundary $\partial T$]
    The hyperbolic boundary $\partial T$ of a tree $(T,d)$ is the set of equivalent classes of $\calc$ modulo the relation \ref{rel}.
    \end{definition}

We can think of two equivalent rays a pair of rays having in common a long sequence of vertex and just differing by a small geodesic segment.
    
By construction of $(T,d)$ with $x_0$ fix, any element $\chi\in\partial T$ has a unique sequence issuing from $x_0$. Write $x_i\to\chi$ for this sequence.

Now we extend Gromov's product to the boundary $\partial T$ using the ideas in \cite{1990} and \cite{doi:10.1112/jlms/54.1.50} defining 
    \begin{equation}\label{rel2}
    (\chi,\chi')=sup\lim_{i,j\to\infty}\{(x_i|y_j)\}
    \end{equation}
where $x_i\to\chi$ and $y_j\to\chi'$ are the unique sequences on their respective equivalence class.

Immediate properties of \ref{rel2} are
    \begin{itemize}
        \item[(1)] For all $\chi,\chi'\in\partial T$, $(\chi|\chi')=1$ if and only if $\chi=\chi'$.
        \item[(2)] For all $\chi,\chi'$, $(\chi,\chi')=(\chi',\chi)$.
        \item[(3)] For all $\chi,\chi',\chi^{''}$, $(\chi,\chi')\geq min\{(\chi,\chi^{''})(\chi^{''},\chi')\}$.
    \end{itemize}
The properties enlisted above are valid for all points in $T$, except for the first property valid only for points in the boundary of $T$. Let $\overline{T}=T\cup\partial T$.

In Figure \ref{figure1} we illustrate a geometric realization of $\overline{T}=T\bigcup\partial T$. We point out the tree $T$ in the circle with center $x_0$. The boundary $\partial T$ is then the circumference. We only highlight points in $\partial T$ in the upper half of the circumference.

Our goal is now to define a distance in $\partial T$ extending $d$ to all $\overline{T}$.

Let $x_i\rightarrow\chi$ and $y_j\rightarrow\chi'$ points in $\partial T$ and $z_{\chi}^{\chi'}$ denote the common vertex of $\{x_i\}$ and $\{y_j\}$ at distance $(\chi|\chi')$ from $x_0$. Denote by $C_{\chi}$ and $C_{\chi'}$ the truncated rays obtained from $\{x_i\}$ and $\{y_j\}$ starting at the common vertex  $z_{\chi}^{\chi'}$ going to $\partial T$ ignoring the previous finite sets. Let $l(C_{\chi})$ and $l(C_{\chi'})$ denote the lengths of the respective rays and define
    \begin{equation*}
    d(\chi,\chi')=l(C_{\chi})+l(C_{\chi'})
    \end{equation*}
this is $d(\chi,\chi')=2(1-(\chi|\chi')$,the projection over a convex set given in \cite{Bartels2008}.

    \begin{figure}[H]
        \centering
        \includegraphics[scale=0.4]{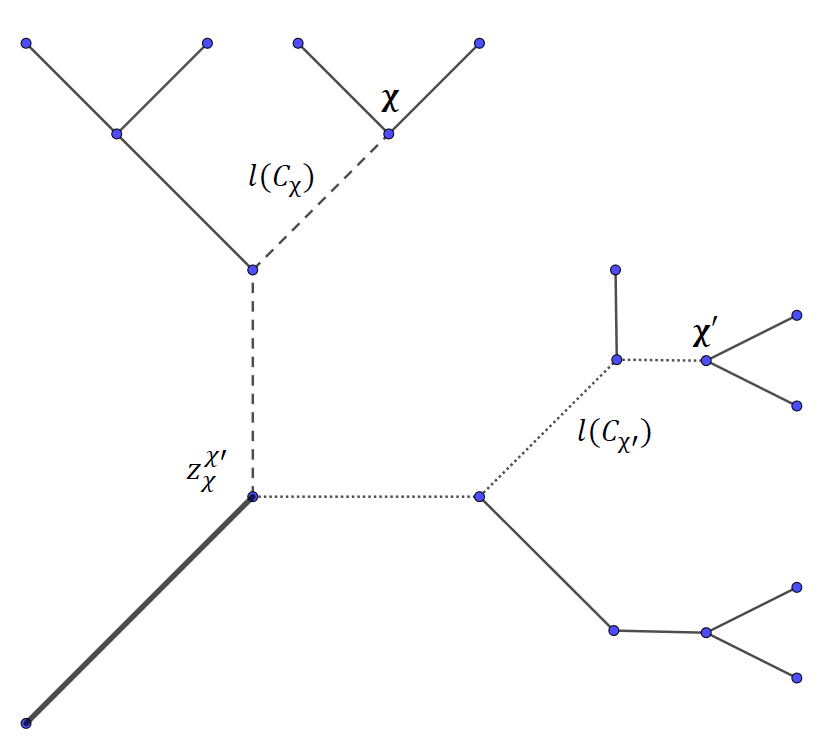}
        \caption{Extended metric on $\overline{T}=T\bigcup\partial T$ using Gromov's product}
        \label{figure2}
    \end{figure}

In Figure \ref{figure2} we illustrate the Gromov product $(\chi|\chi')$ of two finite rays. The point $z_{\chi}^{\chi'}$ marks what vertices are common. In this example, if assume that vertex adjacent to $z_{\chi}^{\chi'}$ has length $1/2^n$, then $l(C_{\chi})$ has length $1/2^n+1/2^{n+1}$ and $l(C_{\chi'})$ has length $1/2^n+1/2^{n+1}+1/2^{n+2}$. Thus $d(\chi,\chi')=l(C_{\chi})+l(C_{\chi'})=1/2^{n-1}+1/2^n+1/2^{n+2}$.

    \begin{proposition}
    The definition $(\partial T,d)$ given above satisfies
        \begin{enumerate}
            \item $d(\chi,\chi')=0$ if and only if $(\chi|\chi')=1$ if and only if $\chi=\chi'$
            \item $d(\chi,\chi')\geq0$ and $d(\chi,\chi')=d(\chi',\chi)$
            \item $d(\chi,\chi')\leq d(\chi,\chi^{''})+d(\chi^{''},\chi')$
            \item for all $\chi\in\partial T$, $d(\chi,x_0)=1$
            \item $d(\chi,\chi')=2$ if $\chi$, $\chi'$ belong to different connected component of $T\setminus\{x_0\}$
        \end{enumerate}
    \end{proposition}

    \begin{remark}
    Formally, we have defined $d$ only for points in the boundary. However, $d(\chi,x_0)$ is defined in the same way we defined $d$ on $\partial T$ but consider finite sequences for points in $T$.
    \end{remark}

\section{Trees are 1-Transfer}

In this section we prove that the space $\overline{T}$ is a $1$-transfer. 

    \begin{proposition}
    The space $\overline{T}$ given the cone topology \cite{Bartels2008} is compact and compatible with $(\overline{T},d)$ given as before.
    \end{proposition}

    \begin{proof}
    The neighbourhood basis given in \cite{Bartels2008} coincides with a basis of balls in the metric $d$.
    \end{proof}

    \begin{proposition}
    The space $(\overline{T},d)$ is contractible.
    \end{proposition}

    \begin{proof}
    In \cite{doi:10.1112/jlms/54.1.50} notation, each ball $\overline{B}(r,x_0)$ centered at $x_0$, with $r\in[0,2]$ is a sub-continuum and hence a tree. 
    \end{proof}
    
Finally, we construct the homotopy between $\overline{T}$ and a $1$-simplex $K$ and verify the control conditions. Let $\delta>0$, hence $1-\delta<1$ and because $\lim_{n\to\infty}(\sum_{i}^{n}1/2^i)=1$ there exists $N$ such that
    \begin{equation*}
        1-\delta\leq \sum_{i=1}^{N}\frac{1}{2^i}\leq 1
    \end{equation*}
denote $\sigma_N=\sum_{i=1}^{N+1}\frac{1}{2^i}$. Using the structure maps, necessaries for the construction of the cone topology as a direct limit, we have a map
    \begin{equation*}
        P_{\sigma_N}:\overline{T}\rightarrow\overline{B}(\sigma_N,x_0).
    \end{equation*}
If we consider $\overline{T}=K$ as $1$-simplex then we obtain
    \begin{equation*}
        \overline{T}\xrightarrow[]{i}K \xrightarrow[]{P_{\sigma_N}}\overline{T}
    \end{equation*}
which is a contraction of $K$ over the subtree $\tau=\overline{B}(\sigma_N,x_0)$. According to \cite{Paulowich1982} there is a homotopy $H$ between $\tau$ and $\overline{T}$ such that $H(x,0)=Id_{\overline{T}}$ and $H(x,1)=P_{\sigma_N}$.

Moreover, for each $x\in\tau$ we have $\{H(x,t)|t\in[0,1]\}=\{x\}$ and hence it has diameter $0$. For any $x\in\overline{T}\setminus\tau$ we have $H(x,0)=x$ and $H(x,1)=P_{\sigma_N}(x)$ as extreme points of the homotopy. Since $P_{\sigma_N}(x)\in\tau$ we have
    \begin{equation*}
        d(P_{\sigma_N,x},x)=1-\sigma_N
    \end{equation*}
and $1-\sigma_n\leq\delta$ we have $diam\{H(t,x)|t\in[0,1]\}\leq\delta$. Hence $\overline{T}$ is a $1$-transfer space.

\setcounter{section}{-1}

\bibliographystyle{plain}
\bibliography{bibliography}

\end{document}